\documentclass[12pt]{article}

\usepackage{times}
\usepackage[a4paper,text={128mm,185mm}, centering]{geometry}
\usepackage{changepage}
\usepackage[english]{babel}
\usepackage{fancyhdr}
\usepackage{amsmath}
\usepackage{amssymb}
\usepackage{amsthm}
\usepackage{amscd}
\usepackage{mathrsfs}
\usepackage[all,cmtip]{xy}
\xyoption{2cell}
\UseAllTwocells
\usepackage{enumerate}
\usepackage{makeidx}
\usepackage[applemac]{inputenc}
\usepackage{array}
\usepackage{multirow}
\usepackage{enumitem}
\usepackage{titlesec}

\titleformat{\section}[hang]%
{\bfseries\Large}{\thesection.}{1ex}{}%

\titleformat{\subsection}[hang]%
{\bfseries\large}{\thesubsection}{1ex}{}%

\newtheorem{theorem}			     {Theorem}	    [section]
\newtheorem{proposition}  [theorem]	 {Proposition}	
\newtheorem{corollary}	  [theorem]	 {Corollary}	
		
\theoremstyle{definition}

\newtheorem{remark} 	  [theorem]  {Remark}
\newtheorem{example}	  [theorem]  {Example}

\newcommand{\C}{\mathcal{C}}
\newcommand{\CC}{\mathbb{C}}
\newcommand{\VV}{\mathbb{V}}

\newcommand{\Mal}{\mathsf{Mal}}
\newcommand{\Perm}{\mathsf{Perm}}
\newcommand{\Maj}{\mathsf{Maj}}
\newcommand{\Ari}{\mathsf{Ari}}
\newcommand{\Cub}{\mathsf{Cube}}
\newcommand{\Edg}{\mathsf{Edge}}
\newcommand{\lex}{\mathsf{lex}}
\newcommand{\reg}{\mathsf{reg}}
\newcommand{\alg}{\mathsf{alg}}
\newcommand{\essalg}{\mathsf{ess\, alg}}
\newcommand{\regessalg}{\mathsf{reg\, ess\, alg}}
\newcommand{\Mod}{\mathsf{Mod}}
\newcommand{\Def}{\mathsf{Def}}
\newcommand{\Set}{\mathbf{Set}}
\newcommand{\Cat}{\mathbf{Cat}}
\newcommand{\op}{\mathsf{op}}
\newcommand{\pb}[1][dr]{\save*!/#1-1.5pc/#1:(-1,1)@^{|-}\restore}

\title{\vskip 5pt  \bf  WHEN A MATRIX CONDITION IMPLIES THE MAL'TSEV PROPERTY}
\author{\itshape\bfseries {Michael Hoefnagel and Pierre-Alain Jacqmin}}
\date{}

\begin{document}

\maketitle

\cfoot{}
\thispagestyle{empty}
\vskip 25pt

\begin{adjustwidth}{0.5cm}{0.5cm}
{\small
{\bf R\'esum\'e.} Les conditions matricielles \'etendent les conditions de Mal'tsev lin\'eaires de l'alg\`ebre universelle aux propri\'et\'es d'exactitude en th\'eorie des cat\'egories. Certaines peuvent \^etre \'enonc\'ees dans le contexte finiment complet alors que, en g\'en\'eral, elles peuvent \^etre \'enonc\'ees seulement pour les cat\'egories r\'eguli\`eres. Nous \'etudions quand une telle condition matricielle implique la propri\'et\'e de Mal'tsev. Nos r\'esultats principaux affirment que, pour les deux types de matrices, cette implication est \'equivalente \`a l'implication correspondante restreinte au contexte des vari\'et\'es d'alg\`ebres universelles.\\
{\bf Abstract.} Matrix conditions extend linear Mal'tsev conditions from Universal Algebra to exactness properties in Category Theory. Some can be stated in the finitely complete context while, in general, they can only be stated for regular categories. We study when such a matrix condition implies the Mal'tsev property. Our main results assert that, for both types of matrices, this implication is equivalent to the corresponding implication restricted to the context of varieties of universal algebras.
}
\end{adjustwidth}

\section*{Introduction}

Given a \emph{simple extended matrix of variables}
$$M=\left[ \begin{array}{ccc|c}
x_{11} & \cdots & x_{1m} & y_{1}\\
\vdots &  & \vdots & \vdots\\
x_{n1} & \cdots & x_{nm} & y_{n}\end{array} \right]$$
where the $x_{ij}$'s and the $y_i$'s are (not necessarily distinct) variables from $\{x_1,\dots,x_k\}$, one can associate the following linear Mal'tsev condition (in the sense of~\cite{Snow}) on a variety of universal algebras~$\VV$: the algebraic theory of $\VV$ admits an $m$-ary term $p$ such that, for each $i\in\{1,\dots,n\}$, the equation $$p(x_{i1},\dots,x_{im}) = y_i$$ holds in~$\VV$. As shown in~\cite{Janelidze1}, this Mal'tsev condition is equivalent to the condition that, for each homomorphic $n$-ary relation $R\subseteq A^n$ on an algebra $A$ of~$\VV$, given any function $f\colon\{x_1,\dots,x_k\}\to A$ interpreting the variables in~$A$, the implication
$$\left\{\left[\begin{array}{c}f(x_{11}) \\ \vdots\\ f(x_{n1})\end{array}\right],\dots,\left[\begin{array}{c}f(x_{1m}) \\ \vdots\\ f(x_{nm})\end{array}\right]\right\}\subseteq R \Longrightarrow \left[\begin{array}{c}f(y_{1}) \\ \vdots\\ f(y_{n})\end{array}\right] \in R$$
holds. While the above linear Mal'tsev condition does not make sense in an arbitrary category, the above condition on relations can be stated in any finitely complete category $\CC$ using internal relations and generalized elements. If this condition is satisfied, we say that $\CC$ \emph{has $M$-closed relations}.

One of the most famous examples of such a condition is given by the matrix
$$\Mal = \left[ \begin{array}{ccc|c}
x_1 & x_2 & x_2 & x_1\\
x_1 & x_1 & x_2 & x_2 \end{array} \right].$$
A finitely complete category has $\Mal$-closed relations if and only if it is a Mal'tsev category~\cite{CPP}, i.e., if and only if every binary internal relation is difunctional in the sense of~\cite{Riguet}. A variety $\VV$ has $\Mal$-closed relations if and only if its theory admits a ternary operation $p$ satisfying the axioms $p(x_1,x_2,x_2)=x_1$ and $p(x_1,x_1,x_2)=x_2$, i.e., if and only if the composition of congruences on any algebra in $\VV$ is commutative~\cite{Maltsev}.

In~\cite{HJJ}, we have described an algorithm to decide whether one matrix condition implies another one in the finitely complete context, i.e., given two simple extended matrices $M_1$ and~$M_2$, whether each finitely complete category with $M_1$-closed relations has $M_2$-closed relations, which we denote by $M_1\Rightarrow_{\lex} M_2$. We have also shown that this algorithm cannot be used in the varietal context. That is, the statement $M_1\Rightarrow_{\lex} M_2$ is in general stronger than the statement that any variety with $M_1$-closed relations has $M_2$-closed relations, which we abbreviate as $M_1\Rightarrow_{\alg} M_2$. Moreover, a general algorithm to decide $M_1\Rightarrow_{\alg} M_2$ still does not exist. However, the results of~\cite{Oprsal} can be used to extract an algorithm for some matrices~$M_2$, including the Mal'tsev matrix~$\Mal$. Surprisingly, in the case $M_2=\Mal$, this algorithm reduces to the algorithm from~\cite{HJJ} for $M_1\Rightarrow_{\lex} \Mal$. This thus means that $M_1\Rightarrow_{\lex} \Mal$ is equivalent to $M_1\Rightarrow_{\alg} \Mal$, which is quite particular to the Mal'tsev matrix~$\Mal$. In that case, the algorithm to decide whether $M_1\Rightarrow_{\lex} \Mal$ reduces to find two (not necessarily distinct) rows of $M_1$ such that, when reducing $M_1$ to those two rows, its right column cannot be found among its left columns. The number of operations required by this algorithm is bounded by a polynomial in the numbers of rows and of columns of the matrix~$M_1$. In addition, using this algorithm and the results of~\cite{HJJ}, we can show that, given a finite number of simple extended matrices $M_1,\dots,M_d$, if each finitely complete category with $M_i$-closed relations for all $i\in\{1,\dots,d\}$ is a Mal'tsev category, then there exists $i\in\{1,\dots,d\}$ such that $M_i\Rightarrow_{\lex}\Mal$.

The linear Mal'tsev conditions arising from simple extended matrices only have equations of the form $$p(x_1,\dots,x_m)=y$$ but not of the form $$p(x_1,\dots,x_m)=p'(x'_1,\dots,x'_{m'})$$ for (not necessarily distinct) variables $x_1,\dots,x_m,x'_1,\dots,x'_{m'}$ and~$y$. In order to take this second kind of equation into account, one needs to consider (not necessarily simple) \emph{extended matrices of variables}
$$M=\left[ \begin{array}{ccc|ccc}
x_{11} & \cdots & x_{1m} & y_{11} & \cdots & y_{1m'}\\
\vdots &  & \vdots & \vdots && \vdots\\
x_{n1} & \cdots & x_{nm} & y_{n1} & \cdots & y_{nm'} \end{array} \right]$$
as introduced in~\cite{Janelidze5}, where the $x_{ij}$'s are variables from $\{x_1,\dots,x_l\}$ and the $y_{ij}$'s are variables from $\{x_1,\dots,x_l,\dots,x_k\}$ (where $k\geqslant l$). The linear Mal'tsev condition on a variety $\VV$ associated to such an $M$ is: the algebraic theory of $\VV$ contains $m$-ary terms $p_1,\dots,p_{m'}$ and $l$-ary terms $q_1,\dots,q_{k-l}$ such that, for each $i\in\{1,\dots,n\}$ and each $j\in\{1,\dots,m'\}$,
$$p_j(x_{i1},\dots,x_{im}) = \begin{cases} x_a \quad &\text{if }\,y_{ij}=x_a\in\{x_1,\dots,x_l\}\\ q_{a-l}(x_1,\dots,x_l) \quad &\text{if }\,y_{ij}=x_a\in\{x_{l+1},\dots,x_k\}\end{cases}$$
is an equation in the variables $x_1,\dots,x_l$ that holds in the algebraic theory of~$\VV$. As shown in~\cite{Janelidze5}, this is equivalent to the condition that, for any homomorphic $n$-ary relation $R\subseteq A^n$ on an algebra $A$ of~$\VV$, given any function $f\colon\{x_1,\dots,x_l\}\to A$, the implication
\begin{align*}
&\left\{\left[\begin{array}{c}f(x_{11}) \\ \vdots\\ f(x_{n1})\end{array}\right],\dots,\left[\begin{array}{c}f(x_{1m}) \\ \vdots\\ f(x_{nm})\end{array}\right]\right\}\subseteq R\\ \Longrightarrow \,&\exists f(x_{l+1}),\dots,f(x_k)\in A\,|\, \left\{\left[\begin{array}{c}f(y_{11}) \\ \vdots\\ f(y_{n1})\end{array}\right],\dots,\left[\begin{array}{c}f(y_{1m'}) \\ \vdots\\ f(y_{nm'})\end{array}\right]\right\}\subseteq R
\end{align*}
holds. In view of the existential quantifier in the above formula, a natural categorical context in which to extend this condition is the context of regular categories in the sense of~\cite{BGV}. In addition to the examples of matrix properties mentioned above, one has now also the example of $n$-permutable categories~\cite{CKP}. The exactness properties on a regular category being expressible by finite conjunctions of such matrix conditions have been semantically characterized in~\cite{JJ2}.

Given two such extended matrices $M_1$ and~$M_2$, a general algorithm to decide whether each regular category with $M_1$-closed relations has $M_2$-closed relations, denoted as $M_1\Rightarrow_{\reg} M_2$, is yet to be found. However, using the embedding theorems from~\cite{JacqminPhD,Jacqmin3}, the statement $M_1\Rightarrow_{\reg} M_2$ is equivalent (assuming the \emph{axiom of universes}~\cite{AGV}) to the statement, denoted by $M_1\Rightarrow_{\regessalg} M_2$, that any regular essentially algebraic category (in the sense of~\cite{AHR,AR}) with $M_1$-closed relations has $M_2$-closed relations. Using this equivalence and the results from~\cite{Oprsal}, we could prove that, when $M_2=\Mal$, the statement $M_1\Rightarrow_{\reg} \Mal$ is equivalent to the statement $M_1\Rightarrow_{\alg} \Mal$.

Our two main theorems, the first one stating the equivalence of $M_1\Rightarrow_{\lex} \Mal$ and $M_1\Rightarrow_{\alg} \Mal$ for a simple extended matrix $M_1$ and the second one stating the equivalence of $M_1\Rightarrow_{\reg} \Mal$ and $M_1\Rightarrow_{\alg} \Mal$ for a (general) extended matrix $M_1$ are quite surprising and particular to the Mal'tsev case. Indeed, as it is the general philosophy of the papers \cite{JacqminPhD,Jacqmin1,Jacqmin2,Jacqmin3,JJ,JJ2,JR}, to prove the validity of many statements about exactness properties, one is often required to produce a proof in the essentially algebraic context (and not just in the varietal context as it is the case in the present situation). Actually, we prove these two theorems not only for the Mal'tsev matrix $\Mal$, but for the matrix $\Cub_n$ for each $n\geqslant 2$, describing the Mal'tsev condition of having an $n$-cube term~\cite{BIMMVW}. The Mal'tsev case is then recovered in the case $n=2$.

This paper is organized as follows. In Section~\ref{section preliminaries}, we recall the necessary material from other papers. In particular, we explain the theory of matrix conditions, in the finitely complete, regular and varietal contexts. We also recall the algorithm from~\cite{HJJ} to decide for an implication $M_1\Rightarrow_{\lex} M_2$ in the finitely complete context and conclude the section with a reminder on essentially algebraic theories. Section~\ref{section main results} contains the main new results of the paper and is divided in two parts. In the first one, we prove Theorem~\ref{theorem simple matrix} which states that given a simple extended matrix $M$ and an integer $n\geqslant 2$, the statement $M\Rightarrow_{\lex}\Cub_n$ is equivalent to $M\Rightarrow_{\alg}\Cub_n$. We also obtain an easy algorithm to decide when these conditions hold. From this algorithm, we deduce (Theorem~\ref{theorem intersection}) that a finite conjunction of conditions induced by simple extended matrices implies the Mal'tsev property if and only if one of these matrix conditions alone already implies the Mal'tsev property. The second part of Section~\ref{section main results} deals with (general) extended matrices and we prove that for such a matrix $M$ and an integer $n\geqslant 2$, the statement $M\Rightarrow_{\reg}\Cub_n$ is equivalent to $M\Rightarrow_{\alg}\Cub_n$ (see Theorem~\ref{theorem general matrix}).

\section*{Acknowledgments}

The authors would like to warmly thank Jakub Opr\v{s}al for fruitful discussions on the subject, without which this paper would not exist. The first author would like to thank the NGA(MaSS) for its financial support. The second author is also grateful to the FNRS for its generous support.

\section{Preliminaries}\label{section preliminaries}

By a \emph{variety}, we mean a one-sorted finitary variety of universal algebras. By a \emph{regular} category, we mean a regular category in the sense of~\cite{BGV}, i.e., a finitely complete category with coequalizers of kernel pairs and pullback stable regular epimorphisms. Regular categories have been introduced as a context where finite limits and regular epimorphisms behave in a similar way as finite limits and surjections behave in the category of sets. In particular, every variety is a regular category. By a \emph{pointed} category, we mean a category with a \emph{zero object}, i.e., an object which is both terminal and initial. A variety is pointed if and only if its algebraic theory contains a unique constant term.

\subsection*{Matrix conditions}

Let us start by recalling the theory of matrix conditions on finitely complete and regular categories as introduced in~\cite{Janelidze1,Janelidze2,Janelidze5}. An \emph{extended matrix $M$ of variables} (or simply an \emph{extended matrix} for short) is given by integer parameters $n \geqslant 1$, $m\geqslant 0$, $m' \geqslant 0$ and $k\geqslant l \geqslant 0$ and by a $n \times (m+m')$ matrix
\begin{equation}\label{equ matrix}
\left[ \begin{array}{ccc|ccc}
x_{11} & \cdots & x_{1m} & y_{11} & \cdots & y_{1m'}\\
\vdots &  & \vdots & \vdots && \vdots\\
x_{n1} & \cdots & x_{nm} & y_{n1} & \cdots & y_{nm'} \end{array} \right]
\end{equation}
where the $x_{ij}$'s are (not necessarily distinct) variables from $\{x_1,\dots,x_l\}$ and the $y_{ij}$'s are (not necessarily distinct) variables from $\{x_1,\dots,x_l,\dots,x_k\}$. When the parameters $n,m,m',l,k$ are clear from the context, we will omit them and we will represent an extended matrix $M$ just by its matrix part; this will be the case when the conditions $m+m'>0$,
$$\{x_{ij}\,|\,i\in\{1,\dots,n\},j\in\{1,\dots,m\}\} = \{x_1,\dots,x_l\}$$
and
$$\{x_1,\dots,x_l\} \cup \{y_{ij}\,|\,i\in\{1,\dots,n\},j\in\{1,\dots,m'\}\} = \{x_1,\dots,x_k\}$$
are all satisfied. The first $m$ columns of $M$ will be called its \emph{left columns}, while its last $m'$ columns will be called its \emph{right columns}. Given an object $A$ in a finitely complete category~$\CC$, each variable $x$ in $\{x_1,\dots,x_l\}$ gives rise to the corresponding projection $x^A\colon A^l\to A$ from the $l$-th power of~$A$ (and similarly, each variable $x$ in $\{x_1,\dots,x_k\}$ gives rise to the corresponding projection $x^A\colon A^k\to A$). Given such an extended matrix~$M$, an $n$-ary internal relation $r\colon R \rightarrowtail A^n$ in a regular category~$\CC$ is said to be \emph{$M$-closed} if, when we consider the pullbacks
$$\xymatrix{P \pb \ar[rrr]^-{f'} \ar@{}[dd]|(.13)*{}="A" \ar@{>->} "A";[dd]_-{f} &&& R^m \ar@{}[dd]|(.13)*{}="B" \ar@<-3pt>@{>->} "B";[dd]^-{r^m} \\ \\
A^l \ar[rrr]_-{\left[\begin{smallmatrix}
x_{11}^A & \cdots & x_{1m}^A\\
\vphantom{\int\limits^x}\smash{\vdots} &  & \vphantom{\int\limits^x}\smash{\vdots} \\
x_{n1}^A & \cdots & x_{nm}^A
\end{smallmatrix} \right]} &&& (A^n)^m} \quad\qquad
\xymatrix{Q \pb \ar[rrr]^-{g'} \ar@{}[dd]|(.13)*{}="A" \ar@{>->} "A";[dd]_-{g} &&& R^{m'} \ar@{}[dd]|(.13)*{}="B" \ar@<-3pt>@{>->} "B";[dd]^-{r^{m'}} \\ \\
A^k \ar[rrr]_-{\left[\begin{smallmatrix}
y_{11}^A & \cdots & y_{1m'}^A\\
\vphantom{\int\limits^x}\smash{\vdots} &  & \vphantom{\int\limits^x}\smash{\vdots} \\
y_{n1}^A & \cdots & y_{nm'}^A
\end{smallmatrix} \right]} &&& (A^n)^{m'}}$$
and
$$\xymatrix{T \ar@{}[r]|(.13)*{}="A" \ar@{>->} "A";[r]^-{h'} \ar[dd]_-{h} \pb & Q \ar@{}[d]|(.25)*{}="C" \ar@{>->} "C";[d]^(.25){g} \\
& A^k \cong A^l \times A^{k-l}  \ar[d]^-{\pi_1 = (x_1^A,\dots,x_l^A)} \\
P \ar@{}[r]|(.13)*{}="B" \ar@{>->} "B";[r]_-{f} & A^l}$$
then $h$ is a regular epimorphism (or, in other words, $f$ factors through the image of $\pi_1 g$). We say that the regular category $\CC$ \emph{has $M$-closed relations} if any internal $n$-ary relation $r\colon R \rightarrowtail A^n$ in $\CC$ is $M$-closed. If $\CC=\VV$ is a variety, an internal relation is a homomorphic relation. An $n$-ary homomorphic relation $R\subseteq A^n$ on an algebra $A$ is $M$-closed when, for each function $f\colon\{x_1,\dots,x_l\}\to A$ such that
$$\left\{\left[\begin{array}{c}f(x_{11}) \\ \vdots\\ f(x_{n1})\end{array}\right],\dots,\left[\begin{array}{c}f(x_{1m}) \\ \vdots\\ f(x_{nm})\end{array}\right]\right\}\subseteq R,$$
there exists an extension $g\colon\{x_1,\dots,x_k\}\to A$ of~$f$ (i.e., $g(x_i)=f(x_i)$ for each $i\in\{1,\dots,l\}$) such that
$$\left\{\left[\begin{array}{c}g(y_{11}) \\ \vdots\\ g(y_{n1})\end{array}\right],\dots,\left[\begin{array}{c}g(y_{1m'}) \\ \vdots\\ g(y_{nm'})\end{array}\right]\right\}\subseteq R.$$
This description can be used to prove the following theorem characterizing varieties with $M$-closed relations via a linear Mal'tsev condition.

\begin{theorem}[\cite{Janelidze5}]\label{theorem characterization varieties M-closed relations}
Let $M$ be an extended matrix as in~(\ref{equ matrix}). A variety $\VV$ has $M$-closed relations if and only if the algebraic theory of $\VV$ contains $m$-ary terms $p_1,\dots,p_{m'}$ and $l$-ary terms $q_1,\dots,q_{k-l}$ such that, for each $i\in\{1,\dots,n\}$ and each $j\in\{1,\dots,m'\}$,
$$p_j(x_{i1},\dots,x_{im}) = \begin{cases} x_a \quad &\text{if }\,y_{ij}=x_a\in\{x_1,\dots,x_l\}\\ q_{a-l}(x_1,\dots,x_l) \quad &\text{if }\,y_{ij}=x_a\in\{x_{l+1},\dots,x_k\}\end{cases}$$
is a theorem of the algebraic theory of $\VV$ in the variables $x_1,\dots,x_l$.
\end{theorem}

For such a matrix~$M$, we will denote by $\VV_M$ the variety whose basic operations are the $m$-ary terms $p_1,\dots,p_{m'}$ and the $l$-ary terms $q_1,\dots,q_{k-l}$ and whose axioms are the theorems described in Theorem~\ref{theorem characterization varieties M-closed relations}. Obviously, $\VV_M$ has $M$-closed relations.

\subsection*{Simple matrix conditions}

An extended matrix $M$ as above will be said to be \emph{simple} when $k=l$ and $m'=1$. We can display such a matrix $M$ as
\begin{equation}\label{equ simple matrix}
\left[ \begin{array}{ccc|c}
x_{11} & \cdots & x_{1m} & y_{1}\\
\vdots &  & \vdots & \vdots\\
x_{n1} & \cdots & x_{nm} & y_{n}\end{array} \right]
\end{equation}
where the $x_{ij}$'s and the $y_i$'s are variables from $\{x_1,\dots,x_k\}$. In that case, the notion of $n$-ary $M$-closed relations can be extended to the finitely complete context as follows. An $n$-ary internal relation $r\colon R\rightarrowtail A^n$ in a finitely complete category $\CC$ is said to be \emph{$M$-closed} when, given any object $B$ and any function $f\colon\{x_1,\dots,x_k\}\to\CC(B,A)$ such that the induced morphism
$$\left[\begin{array}{c}f(x_{1j}) \\ \vdots\\ f(x_{nj})\end{array}\right]\colon B\to A^n$$
factor through $r$ for each $j\in\{1,\dots,m\}$, then so does the morphism
$$\left[\begin{array}{c}f(y_{1}) \\ \vdots\\ f(y_{n})\end{array}\right]\colon B\to A^n.$$
For a simple extended matrix~$M$, we say that the finitely complete category $\CC$ has \emph{$M$-closed relations} when each internal $n$-ary relation $r\colon R\rightarrowtail A^n$ is $M$-closed. If $\CC=\VV$ is a variety, a homomorphic relation $R\subseteq A^n$ is $M$-closed when, for each function $f\colon\{x_1,\dots,x_k\}\to A$, the implication
$$\left\{\left[\begin{array}{c}f(x_{11}) \\ \vdots\\ f(x_{n1})\end{array}\right],\dots,\left[\begin{array}{c}f(x_{1m}) \\ \vdots\\ f(x_{nm})\end{array}\right]\right\}\subseteq R \Longrightarrow \left[\begin{array}{c}f(y_{1}) \\ \vdots\\ f(y_{n})\end{array}\right] \in R$$
holds. Particularizing Theorem~\ref{theorem characterization varieties M-closed relations} to this simpler situation, one gets the following.

\begin{theorem}[\cite{Janelidze1}]\label{theorem characterization varieties simple M-closed relations}
Let $M$ be a simple extended matrix as in~(\ref{equ simple matrix}). A variety $\VV$ has $M$-closed relations if and only if the algebraic theory of $\VV$ contains an $m$-ary term $p$ such that, for each $i\in\{1,\dots,n\}$,
$$p(x_{i1},\dots,x_{im}) = y_i$$
is a theorem of the algebraic theory of $\VV$ in the variables $x_1,\dots,x_k$.
\end{theorem}

Before describing some examples of matrix conditions, let us introduce some notation. We denote by $\lex$ (respectively by $\lex_\ast$, $\reg$, $\reg_\ast$, $\alg$ and $\alg_\ast$) the collection of finitely complete categories (respectively of finitely complete pointed categories, regular categories, regular pointed categories, varieties and pointed varieties). The notation $\lex$ abbreviates `left exact categories' which is another name for finitely complete categories. Given two extended matrices $M_1$ and~$M_2$ and a sub-collection $\C$ of $\reg$ (respectively, two simple extended matrices $M_1$ and~$M_2$ and a sub-collection $\C$ of $\lex$), we write $M_1\Rightarrow_{\C} M_2$ to mean that any category in $\C$ with $M_1$-closed relations has $M_2$-closed relations. We write $M_1\Leftrightarrow_{\C} M_2$ for the conjunction of the statements $M_1\Rightarrow_{\C} M_2$ and $M_2\Rightarrow_{\C} M_1$. We also write $M_1\nRightarrow_{\C} M_2$ for the negation of the statement $M_1\Rightarrow_{\C} M_2$.

\subsection*{Examples}

\begin{example}\label{example Mal'tsev}
Let $\Mal$ be the simple extended matrix given by
$$\Mal = \left[ \begin{array}{ccc|c}
x_1 & x_2 & x_2 & x_1\\
x_1 & x_1 & x_2 & x_2 \end{array} \right].$$
A finitely complete category has $\Mal$-closed relations if and only if it is a Mal'tsev category as introduced in~\cite{CPP}. A variety has $\Mal$-closed relations if and only if its theory admits a Mal'tsev term, i.e., if and only if it is $2$-permutable~\cite{Maltsev}.
\end{example}

\begin{example}
More generally, for any $r\geqslant 2$, let $\Perm_r$ be the extended matrix given by
$$\Perm_r = \left[ \begin{array}{ccc|ccccc}
x_1 & x_2 & x_2 & x_1 & x_3 & x_4 & \cdots & x_r\\
x_1 & x_1 & x_2 & x_3 & x_4 & \cdots & x_r & x_2 \end{array} \right].$$
A regular category has $\Perm_r$-closed relations if and only if it is an $r$-per\-mut\-able category as introduced in~\cite{CKP}, generalizing the notion of an $r$-permutable variety.
\end{example}

\begin{example}
Let $\Ari$ be the simple extended matrix given by
$$\Ari = \left[ \begin{array}{ccc|c}
x_1 & x_2 & x_2 & x_1\\
x_1 & x_1 & x_2 & x_2\\
x_1 & x_2 & x_1 & x_1 \end{array} \right].$$
The notion of a finitely complete category with $\Ari$-closed relations extends to the finitely complete context the notions of an arithmetical category in the sense of~\cite{Pedicchio} and of an equivalence distributive Mal'tsev category in the sense of~\cite{GRT}. A variety $\VV$ has $\Ari$-closed relations if and only if its theory admits a Pixley term~\cite{Pixley}.
\end{example}

\begin{example}
Let $\Maj$ be the simple extended matrix given by
$$\Maj = \left[ \begin{array}{ccc|c}
x_1 & x_1 & x_2 & x_1\\
x_1 & x_2 & x_1 & x_1\\
x_2 & x_1 & x_1 & x_1 \end{array} \right].$$
A finitely complete category has $\Maj$-closed relations if and only if it is a majority category as introduced in~\cite{Hoefnagel}. A variety $\VV$ is a majority category if and only if its theory admits a majority term.
\end{example}

\begin{example}\label{example cube terms}
For any $n,k \geqslant 2$, let $\Cub_{n,k}$ be the simple extended matrix with $n$ rows, $k^n-1$ left columns, one right column and $k=l$ variables defined by taking as left columns (ordered lexicographically) all possible $n$-tuples of elements of $\{x_1,\dots,x_k\}$ except $(x_1,\dots,x_1)$, which is used as right column. As we shall see below (see Corollary~\ref{corollary cubical matrices}), for any $n,k_1,k_2\geqslant 2$, one has $\Cub_{n,k_1} \Leftrightarrow_{\lex} \Cub_{n,k_2}$. In view of this, we abbreviate $\Cub_{n,2}$ by $\Cub_n$. Up to permutation of left columns and change of variables, $\Cub_2$ is the matrix $\Mal$ from Example~\ref{example Mal'tsev}, and therefore, $\Cub_2 \Leftrightarrow_{\lex} \Mal$. The matrix $\Cub_3$ is the matrix
$$\Cub_3 = \left[ \begin{array}{ccccccc|c}
x_1 & x_1 & x_1 & x_2 & x_2 & x_2 & x_2 & x_1\\
x_1 & x_2 & x_2 & x_1 & x_1 & x_2 & x_2 & x_1\\
x_2 & x_1 & x_2 & x_1 & x_2 & x_1 & x_2 & x_1 \end{array} \right].$$
For $n\geqslant 2$, a variety has $\Cub_n$-closed relations if and only if its theory admits a $n$-cube term in the sense of~\cite{BIMMVW}.
\end{example}

\begin{example}\label{example edge terms}
For any $n\geqslant 2$, let $\Edg_n$ be the simple extended matrix with $n$ rows, $n+1$ left columns, one right column and $k=l=2$ variables defined by
$$\Edg_n = \left[\begin{array}{cccccc|c}
x_2 & x_2 & x_1 & x_1 & \cdots & x_1 & x_1\\
x_2 & x_1 & x_2 & x_1 & \cdots & x_1 & x_1\\
x_1 & x_1 & x_1 & x_2 & \cdots & x_1 & x_1\\
\vdots & \vdots & \vdots & \vdots & \ddots & \vdots & \vdots\\
x_1 & x_1 & x_1 & x_1 & \cdots & x_2 & x_1 
\end{array}\right]$$
where the entries in positions $(1,1)$, $(2,1)$ and $(i,i+1)$ for all $i\in\{1,\dots,n\}$ are $x_2$'s and the other ones are~$x_1$'s. Up to permutation of left columns and change of variables, $\Edg_2$ is the matrix $\Mal$, and thus $\Edg_2 \Leftrightarrow_{\lex} \Mal$. For a general $n\geqslant 2$, a variety has $\Edg_n$-closed relations if and only if its theory admits an $n$-edge term in the sense of~\cite{BIMMVW}. Therefore, as it is shown in~\cite{BIMMVW}, one has $\Cub_n\Leftrightarrow_{\alg}\Edg_n$. Using Proposition~1.7 of~\cite{Janelidze2}, we know that $\Edg_n \Rightarrow_{\lex} \Cub_n$. However, for a general~$n$, the converse is not true. For instance, from the computer-based results of~\cite{HJJ}, we can see that $\Cub_3\nRightarrow_{\lex}\Edg_3$. This is an additional example of the context dependency of the implications between matrix properties.
\end{example}

\subsection*{The algorithm in the finitely complete context}

Given two extended matrices $M_1$ and~$M_2$, as far as we know, algorithms to decide whether $M_1\Rightarrow_{\reg} M_2$ or whether $M_1\Rightarrow_{\alg} M_2$ do not exist yet. On the contrary, in~\cite{HJJ}, an algorithm to decide whether $M_1\Rightarrow_{\lex} M_2$ for two simple extended matrices $M_1$ and $M_2$ has been developed. Since we will need it further, let us recall it here.

A simple exended matrix $M$ is called \emph{trivial} if any finitely complete category with $M$-closed relations is a \emph{preorder} (i.e., a category with at most one morphism between any two objects). In order to state their characterization from~\cite{HJJ}, we need the following notation, using the presentation of $M$ as in~(\ref{equ simple matrix}). Given $i\in\{1,\dots,n\}$, $R_{M_i}$ denotes the equivalence relation on the set $\{1,\dots,m\}$ defined by $j_1 R_{M_i} j_2$ if and only if $x_{ij_1}=x_{ij_2}$. Given two equivalence relations $R$ and $S$ on the same set, $R\vee S$ denotes the smallest equivalence relation containing both $R$ and~$S$. Finally, we denote by $\Set^{\op}$ the dual of the category $\Set$ of sets.

\begin{theorem}[\cite{HJJ}]\label{theorem trivial matrices}
For a simple extended matrix $M$ as in~(\ref{equ simple matrix}), the following conditions are equivalent:
\begin{enumerate}[label=(\alph*), ref=(\alph*)]
\item $M$ is \emph{not} a trivial matrix.
\item $\Set^{\op}$ has $M$-closed relations.
\item For all $i,i'\in\{1,\dots,n\}$, there exist $j,j'\in\{1,\dots,m\}$ such that $x_{ij}=y_i$, $x_{i'j'}=y_{i'}$ and $j R_{M_i} \vee R_{M_{i'}} j'$.
\end{enumerate}
\end{theorem}

Let $M_1$ and $M_2$ be two simple extended matrices with parameters $n_1$, $m_1$, $m'_1=1$, $k_1=l_1$ and $n_2$, $m_2$, $m'_2=1$, $k_2=l_2$. We know that if $m_1=0$, then $M_1$ is trivial, we always have $M_1\Rightarrow_{\lex}M_2$ and we have $M_2\Rightarrow_{\lex}M_1$ if and only if $m_2=0$. If $m_1>0$ and $M_1$ is trivial, then we have $M_1\Rightarrow_{\lex}M_2$ if and only if $m_2>0$ and we have $M_2\Rightarrow_{\lex}M_1$ if and only if $M_2$ is trivial. It thus remains to explain, in the case where neither $M_1$ nor $M_2$ is trivial, how to decide whether $M_1\Rightarrow_{\lex}M_2$. In order to describe this algorithm, we need the following notion. Given $i\in\{1,\dots,n_1\}$ and a set~$S$, an \emph{interpretation of type $S$} of the $i$-th row of $M_1$ is a $(m_1+1)$-tuple
$$\left[\begin{array}{ccc|c}f(x^1_{i1}) & \dots & f(x^1_{im_1}) & f(y^1_i) \end{array}\right]$$
formed by applying a function $f\colon\{x_1,\dots,x_{k_1}\}\to S$ to the entries of the $i$-th row of~$M_1$. Now, if neither $M_1$ nor $M_2$ is a trivial matrix, the algorithm from~\cite{HJJ} to decide whether $M_1\Rightarrow_{\lex}M_2$ is the following:
\begin{itemize}
\item[] Keep expanding the set of left columns of~$M_2$, until it is no more possible, with right columns of $n_2\times (m_1+1)$ matrices
$$\left[\begin{array}{ccc|c}f_{i_1}(x^1_{i_1 1}) & \dots & f_{i_1}(x^1_{i_1 m_1}) & f_{i_1}(y^1_{i_1}) \\ \vdots && \vdots & \vdots \\ f_{i_{n_2}}(x^1_{i_{n_2} 1}) & \dots & f_{i_{n_2}}(x^1_{i_{n_2} m_1}) & f_{i_{n_2}}(y^1_{i_{n_2}}) \end{array}\right]$$
for which each row is an interpretation of type $\{x_1,\dots,x_{k_2}\}$ of a row of~$M_1$, each of the first $m_1$ left columns is a left column of~$M_2$ but the right column is not. Then $M_1\Rightarrow_{\lex}M_2$ holds if and only if the right column of $M_2$ is contained in the left columns of the expanded matrix~$M_2$.
\end{itemize}

\subsection*{Essentially algebraic categories}

Let us conclude this preliminary section with a reminder on (many-sorted) essentially algebraic categories~\cite{AHR,AR} (or in other words locally presentable categories~\cite{GU}) as we will need them to prove Theorem~\ref{theorem general matrix}. They are described by \emph{essentially algebraic theories}, i.e., quintuples $$\Gamma=(S,\Sigma,E,\Sigma_t,\Def)$$ where
\begin{itemize}
\item $S$ is a set of sorts;
\item $\Sigma$ is an $S$-sorted signature of algebras, i.e., a set of operation symbols $\sigma$ with prescribed arity $\sigma\colon\prod_{u\in U}s_u\to s$ where $U$ is a set, $s_u\in S$ for each $u\in U$ and $s\in S$;
\item $E$ is a set of $\Sigma$-equations;
\item $\Sigma_t$ is a subset of $\Sigma$ called the set of \emph{total operation symbols};
\item $\Def$ is a function assigning to each operation symbol $\sigma\colon\prod_{u\in U}s_u\to s$ in $\Sigma\setminus\Sigma_t$ a set $\Def(\sigma)$ of $\Sigma_t$-equations $t_1=t_2$ where $t_1$ and $t_2$ are $\Sigma_t$-terms $\prod_{u\in U}s_u\to s'$.
\end{itemize}
A \emph{$\Gamma$-model} is an $S$-sorted set $A=(A_s)_{s\in S}$ together with, for each operation symbol $\sigma \colon \prod_{u\in U} s_u \to s$ in~$\Sigma$, a partial function $\sigma^A \colon \prod_{u \in U} A_{s_u} \to A_s$ such that:
\begin{enumerate}[label = \arabic*. , ref=\arabic*]
\item for each $\sigma \in \Sigma_t$, $\sigma^A$ is totally defined;
\item given $\sigma \colon \prod_{u\in U} s_u \to s$ in $\Sigma \setminus \Sigma_t$ and a family $(a_u \in A_{s_u})_{u\in U}$ of elements, $\sigma^A((a_u)_{u\in U})$ is defined if and only if the identity $$t_1^A((a_u)_{u\in U})=t_2^A((a_u)_{u\in U})$$ holds for each $\Sigma_t$-equation of~$\Def(\sigma)$;
\item $A$ satisfies the equations of $E$ wherever they are defined.
\end{enumerate}
 A $\Sigma$-term $t\colon\prod_{u\in U} s_u \to s$ will be said to be \emph{everywhere-defined} if, for each $\Gamma$-model~$A$, the induced function $t^A \colon \prod_{u \in U} A_{s_u} \to A_s$ is totally defined (see~\cite{JacqminPhD,Jacqmin1} for more details). A \emph{homomorphism} $f \colon A \to B$ of $\Gamma$-models is an $S$-sorted function $(f_s\colon A_s\to B_s)_{s\in S}$ such that, given $\sigma\colon \prod_{u \in U} s_u \to s$ in $\Sigma$ and a family $(a_u \in A_{s_u})_{u\in U}$ such that $\sigma^A((a_u)_{u\in U})$ is defined in~$A$, then $\sigma^B((f_{s_u}(a_u))_{u\in U})$ is defined in $B$ and the identity
$$f_s(\sigma^A((a_u)_{u \in U}))=\sigma^B((f_{s_u}(a_u))_{u \in U})$$
holds. The $\Gamma$-models and their homomorphisms form the category $\Mod(\Gamma)$. A category which is equivalent to a category $\Mod(\Gamma)$ for some essentially algebraic theory~$\Gamma$ is called \emph{essentially algebraic}. These are exactly the locally presentable categories. Note that essentially algebraic categories are in general not regular but have a (strong epimorphism, monomorphism)-factorization system. Each variety is a regular essentially algebraic category. The category $\Cat$ of small categories is a non-regular essentially algebraic category. We denote by $\essalg$ (respectively $\essalg_\ast$, $\regessalg$ and $\regessalg_\ast$) the collection of essentially algebraic categories (respectively of essentially algebraic pointed categories, regular essentially algebraic categories and of regular essentially algebraic pointed categories). Therefore, for extended matrices $M_1$ and~$M_2$, the notation $M_1\Rightarrow_{\regessalg}M_2$ means that any regular essentially algebraic category with $M_1$-closed relations has $M_2$-closed relations.

\section{Main results}\label{section main results}

\subsection*{Results for simple matrices}

In the case where $M_2=\Cub_{n,k}$ is the matrix from Example~\ref{example cube terms}, the algorithm for deciding $M_1\Rightarrow_{\lex} M_2$ can be nicely simplified.

\begin{proposition}\label{proposition simple matrix}
Let $M$ be a simple extended matrix (with parameters $n\geqslant 1$, $m \geqslant 0$, $m'=1$ and $k=l\geqslant 1$)
$$M=\left[ \begin{array}{ccc|c}
x_{11} & \cdots & x_{1m} & y_{1}\\
\vdots &  & \vdots & \vdots\\
x_{n1} & \cdots & x_{nm} & y_{n}\end{array} \right]$$
and let $n',k'\geqslant 2$ be integers. The implication $M\Rightarrow_{\lex} \Cub_{n',k'}$ holds if and only if there exist $i_1,\dots,i_{n'}\in\{1\dots,n\}$ such that there does not exist $j\in\{1,\dots,m\}$ for which $x_{i_aj}=y_{i_a}$ for each $a\in\{1,\dots,n'\}$.
\end{proposition}

\begin{proof}
Firstly, let us notice that Theorem~\ref{theorem trivial matrices} implies that $\Cub_{n',k'}$ is not a trivial matrix. Let us now assume that $M$ is trivial. In that case, we have $M\Rightarrow_{\lex}\Cub_{n',k'}$. If the left part of the $i$-th row of~$M$ (for some $i\in\{1,\dots,n\}$) does not contain $y_i$ as an entry, the second condition is satisfied with $i_1=\cdots=i_{n'}=i$. Else, by Theorem~\ref{theorem trivial matrices}, there must exist $i,i'\in\{1,\dots,n\}$ and $j,j'\in\{1,\dots,m\}$ such that $x_{ij}=y_i$, $x_{i'j'}=y_{i'}$ and $j$ is not related to $j'$ by $R_{M_i}\vee R_{M_{i'}}$. Since $n'\geqslant 2$, the second condition in the statement is then satisfied with $i_1=i$ and $i_2=\cdots=i_{n'}=i'$.

We can thus suppose without loss of generality that $M$ is not trivial. In the algorithm to decide whether $M\Rightarrow_{\lex}\Cub_{n',k'}$, there is only one column that could be added to $\Cub_{n',k'}$, i.e., the column of~$x_1$'s. This column can indeed be added if and only if we can find $i_1,\dots,i_{n'}\in\{1,\dots,n\}$ and functions $f_{i_1},\dots,f_{i_{n'}}\colon\{x_1,\dots,x_k\}\to\{x_1,\dots,x_{k'}\}$ such that the left columns of the matrix
$$\left[\begin{array}{ccc|c}f_{i_1}(x_{i_1 1}) & \dots & f_{i_1}(x_{i_1 m}) & f_{i_1}(y_{i_1}) \\ \vdots && \vdots & \vdots \\ f_{i_{n'}}(x_{i_{n'} 1}) & \dots & f_{i_{n'}}(x_{i_{n'} m}) & f_{i_{n'}}(y_{i_{n'}}) \end{array}\right]$$
are different from the $n'$-tuple $(x_1,\dots,x_1)$, but the right column is equal to it. This condition clearly implies the one in the statement. Conversely, from the condition in the statement, one can construct such a matrix by considering, for each $a\in\{1,\dots,n'\}$, the function $f_{i_a}\colon\{x_1,\dots,x_k\}\to\{x_1,\dots,x_{k'}\}$ which sends $y_{i_a}$ to $x_1$ and the other elements of the domain to~$x_2$ (using the fact that $k'\geqslant 2$).
\end{proof}

Since the above condition characterizing $M\Rightarrow_{\lex} \Cub_{n',k'}$ does not depend on~$k'$, one immediately has the following corollary.

\begin{corollary}\label{corollary cubical matrices}
For $n,k_1,k_2\geqslant 2$, one has $\Cub_{n,k_1}\Leftrightarrow_{\lex}\Cub_{n,k_2}$.
\end{corollary}

Putting together Proposition~\ref{proposition simple matrix} and the results of~\cite{Oprsal}, we can easily prove the following theorem. We recall that for an extended matrix~$M$, we have defined after Theorem~\ref{theorem characterization varieties M-closed relations} the variety $\VV_M$ as the `generic' one with $M$-closed relations. If $M$ is simple, $\VV_M$ is thus obtained with a single basic operation and one axiom for each row of $M$ as described in Theorem~\ref{theorem characterization varieties simple M-closed relations}.

\begin{theorem}\label{theorem simple matrix}
Let $M$ be a simple extended matrix (with parameters $n\geqslant 1$, $m \geqslant 0$, $m'=1$ and $k=l\geqslant 1$)
$$M=\left[ \begin{array}{ccc|c}
x_{11} & \cdots & x_{1m} & y_{1}\\
\vdots &  & \vdots & \vdots\\
x_{n1} & \cdots & x_{nm} & y_{n}\end{array} \right]$$
and let $n'\geqslant 2$ be an integer. The following statements are equivalent:
\begin{enumerate}[label=(\alph*), ref=(\alph*)]
\item\label{theorem simple matrix lex} $M\Rightarrow_{\lex}\Cub_{n'}$
\item\label{theorem simple matrix pointed lex} $M\Rightarrow_{\lex_\ast}\Cub_{n'}$
\item\label{theorem simple matrix reg} $M\Rightarrow_{\reg}\Cub_{n'}$
\item\label{theorem simple matrix pointed reg} $M\Rightarrow_{\reg_\ast}\Cub_{n'}$
\item\label{theorem simple matrix alg} $M\Rightarrow_{\alg}\Cub_{n'}$
\item\label{theorem simple matrix pointed alg} $M\Rightarrow_{\alg_\ast}\Cub_{n'}$
\item\label{theorem simple matrix ess alg} $M\Rightarrow_{\essalg}\Cub_{n'}$
\item\label{theorem simple matrix pointed ess alg} $M\Rightarrow_{\essalg_\ast}\Cub_{n'}$
\item\label{theorem simple matrix reg ess alg} $M\Rightarrow_{\regessalg}\Cub_{n'}$
\item\label{theorem simple matrix pointed reg ess alg} $M\Rightarrow_{\regessalg_\ast}\Cub_{n'}$
\item\label{theorem simple matrix existence of rows} There exist $i_1,\dots,i_{n'}\in\{1,\dots,n\}$ such that there does not exist $j\in\{1,\dots,m\}$ for which $x_{i_aj}=y_{i_a}$ for each $a\in\{1,\dots,n'\}$.
\item\label{theorem simple matrix non existence of algebra} There does \emph{not} exist a function $p\colon \{0,1\}^m\to\{0,1\}$ making $(\{0,1\},p)$ an algebra of $\VV_M$ such that its induced $n'$-power $(\{0,1\}^{n'},p^{n'})$ is compatible with the $n'$-ary relation $R_{n'}=\{0,1\}^{n'}\setminus\{(0,\dots,0)\}$ (i.e., $p^{n'}(r_1,\dots,r_m)\in R_{n'}$ for each $r_1,\dots,r_m\in R_{n'}$).
\end{enumerate}
\end{theorem}

\begin{proof}
The equivalence \ref{theorem simple matrix alg}$\Leftrightarrow$\ref{theorem simple matrix non existence of algebra} is an immediate application of Lemma~3.5 and Proposition~7.7 of~\cite{Oprsal} applied to the variety~$\VV_M$. The equivalence \ref{theorem simple matrix lex}$\Leftrightarrow$\ref{theorem simple matrix existence of rows} is Proposition~\ref{proposition simple matrix} with $k'=2$. It is trivial that \ref{theorem simple matrix lex} implies all the statements \ref{theorem simple matrix pointed lex}--\ref{theorem simple matrix pointed reg ess alg} and any of these statements implies \ref{theorem simple matrix pointed alg}. The implication \ref{theorem simple matrix pointed alg}$\Rightarrow$\ref{theorem simple matrix alg} is an immediate application of Theorem~3.3 of~\cite{Janelidze2} (see also~\cite{HJ}). It thus suffices to show the implication \ref{theorem simple matrix non existence of algebra}$\Rightarrow$\ref{theorem simple matrix existence of rows}. By contradiction, let us suppose \ref{theorem simple matrix existence of rows} does not hold and let us prove \ref{theorem simple matrix non existence of algebra} does not hold neither. We define $p\colon\{0,1\}^m\to\{0,1\}$ on an $m$-tuple $(b_1,\dots,b_m)$ by $p(b_1,\dots,b_m)=0$ if and only if there exist $i\in\{1,\dots,n\}$ and a function $f\colon\{x_1,\dots,x_k\}\to\{0,1\}$ such that $f(y_i)=0$ and $(b_1,\dots,b_m)=(f(x_{i1}),\dots,f(x_{im}))$. To prove that $(\{0,1\},p)$ indeed forms an algebra of~$\VV_M$, we need to show that for any $i\in\{1,\dots,n\}$ and any function $f\colon\{x_1,\dots,x_k\}\to\{0,1\}$, one has
$$p(f(x_{i1}),\dots,f(x_{im}))=f(y_i).$$
If $f(y_i)=0$, this is immediate from the definition of~$p$. If $f(y_i)=1$, we need to show there do not exist $i'\in\{1,\dots,n\}$ and $g\colon\{x_1,\dots,x_k\}\to\{0,1\}$ such that $g(y_{i'})=0$ and $(f(x_{i1}),\dots,f(x_{im}))=(g(x_{i'1}),\dots,g(x_{i'm}))$. But if this was the case, using that $n'\geqslant 2$ and that \ref{theorem simple matrix existence of rows} is false with $i_1=i$ and $i_2=\cdots=i_{n'}=i'$, we obtain a $j\in\{1,\dots,m\}$ such that $x_{ij}=y_i$ and $x_{i'j}=y_{i'}$, contradicting $f(x_{ij})=g(x_{i'j})$. It remains to prove that $p^{n'}$ is compatible with $R_{n'}$. The only way for it not to be so is that there exist $i_1,\dots,i_{n'}\in\{1,\dots,n\}$ and functions $f_1,\dots,f_{n'}\colon\{x_1,\dots,x_k\}\to\{0,1\}$ such that $f_a(y_{i_a})=0$ for each $a\in\{1,\dots,n'\}$ and such that the matrix
$$\left[\begin{array}{ccc}
f_1(x_{i_11}) & \cdots & f_1(x_{i_1m})\\
\vdots && \vdots\\
f_{n'}(x_{i_{n'}1}) & \cdots & f_{n'}(x_{i_{n'}m})
\end{array}\right]$$
does not contain a column of~$0$'s. But this is impossible by our assumption that \ref{theorem simple matrix existence of rows} is false.
\end{proof}

The situation described by Theorem~\ref{theorem simple matrix} is very particular to the matrices $\Cub_{n'}$. Indeed, as we have already remarked in Example~\ref{example edge terms}, one has $\Cub_3 \Rightarrow_{\alg} \Edg_3$ but $\Cub_3 \nRightarrow_{\lex} \Edg_3$. Other examples of this phenomenon are given in~\cite{HJJ}.

Since it is our most interesting case, let us specify some of the statements of Theorem~\ref{theorem simple matrix} in the case where $n'=2$, i.e., when $\Cub_{n'}=\Cub_2$ describes the Mal'tsev property.

\begin{corollary}\label{corollary simple matrix Maltsev}
Let $M$ be a simple extended matrix (with parameters $n\geqslant 1$, $m \geqslant 0$, $m'=1$ and $k=l\geqslant 1$)
$$M=\left[ \begin{array}{ccc|c}
x_{11} & \cdots & x_{1m} & y_{1}\\
\vdots &  & \vdots & \vdots\\
x_{n1} & \cdots & x_{nm} & y_{n}\end{array} \right].$$
The following statements are equivalent:
\begin{enumerate}[label=(\alph*), ref=(\alph*)]
\item Each finitely complete category with $M$-closed relations is a Mal'tsev category.
\item Each variety with $M$-closed relations is a Mal'tsev variety.
\item\label{corollary simple matrix Maltsev existence of rows} There exist $i,i'\in\{1\dots,n\}$ such that there is no $j\in\{1,\dots,m\}$ for which $x_{ij}=y_{i}$ and $x_{i'j}=y_{i'}$.
\end{enumerate}
\end{corollary}

\begin{remark}
It is shown in~\cite{HJJW} that if $M$ is a simple extended matrix as in Corollary~\ref{corollary simple matrix Maltsev} with $k=2$, then the equivalent conditions in that corollary are further equivalent to:
\begin{enumerate}[label=(\alph*), ref=(\alph*), resume]
\item There is a finitely complete majority category which does not have $M$-closed relations.
\end{enumerate}
In other words, if $k=2$, $M\Rightarrow_{\lex}\Mal$ if and only if $\Maj\nRightarrow_{\lex}M$. It is also shown that if $k>2$, this equivalence is no longer true.
\end{remark}

Statement~\ref{theorem simple matrix existence of rows} of Theorem~\ref{theorem simple matrix} provides an algorithm to decide whether $M \Rightarrow_{\alg} \Cub_{n'}$ (or equivalently $M \Rightarrow_{\lex} \Cub_{n'}$). It is easy to see that this algorithm requires at most $m\times n^{n'}$ comparisons of columns, and thus at most $n'\times m \times n^{n'}$ comparisons of elements. For a fixed~$n'$, this is thus a polynomial-time algorithm in the parameters $n$ and $m$ of the input matrix~$M$. For $n'=2$, statement~\ref{corollary simple matrix Maltsev existence of rows} of Corollary~\ref{corollary simple matrix Maltsev} thus provides an algorithm to decide whether $M \Rightarrow_{\alg} \Mal$ (or equivalently $M \Rightarrow_{\lex} \Mal$) which requires at most $2mn^2$ comparisons of elements.

Theorem~3.6 in~\cite{HJJ} states that the intersection of finitely many matrix properties induced by simple extended matrices is again a matrix property induced by a simple extended matrix. Given, for each $i\in\{1,2\}$, such a matrix $M_i$ with parameters $n_i$, $m_i$, $m'_i=1$ and $k_i=l_i$, we can form a matrix $M$ with parameters $n=n_1+n_2$, $m=m_1\times m_2$, $m'=1$ and $k=l=\max(k_1,k_2)$ as follows: the left columns of $M$ are indexed by the pairs consisting of a left column of $M_1$ and a left column of $M_2$ and are obtained by superposing this column of $M_1$ over this column of~$M_2$. The right column of $M$ is obtained by superposing the right column of $M_1$ over the right column of~$M_2$. Then, a finitely complete category has $M$-closed relations if and only if it has $M_1$-closed relations and $M_2$-closed relations.

\begin{theorem}\label{theorem intersection}
Let $d\geqslant 0$ be an integer and $(M_i)_{i\in\{1,\dots,d\}}$ be a finite family of simple extended matrices. For $n'\geqslant 2$, the following statements are equivalent:
\begin{enumerate}[label=(\alph*), ref=(\alph*)]
\item\label{theorem intersection all} Each finitely complete category with $M_i$-closed relations for all $i\in\{1,\dots,d\}$ has $\Cub_{n'}$-closed relations.
\item\label{theorem intersection one} There exists $i\in\{1,\dots,d\}$ such that $M_i\Rightarrow_{\lex}\Cub_{n'}$.
\end{enumerate}
\end{theorem}

\begin{proof}
The statement being trivial for $d=0$ and $d=1$, let us assume without loss of generality that $d\geqslant 2$. Furthermore, since the intersection of finitely many matrix properties induced by simple extended matrices is again a matrix property induced by a simple extended matrix, using induction, we can assume without loss generality that $d=2$. The implication \ref{theorem intersection one}$\Rightarrow$\ref{theorem intersection all} being trivial, we assume \ref{theorem intersection all} and we shall prove~\ref{theorem intersection one}. Let $M$ be the simple extended matrix as constructed above from $M_1$ and~$M_2$. We shall use the same notation as above for the parameters of $M_1$, $M_2$ and~$M$. Moreover, for $i\in\{1,2\}$, we denote the entries of $M_i$ as
$$M_i=\left[ \begin{array}{ccc|c}
x^i_{11} & \cdots & x^i_{1m_i} & y^i_{1}\\
\vdots &  & \vdots & \vdots\\
x^i_{n_i1} & \cdots & x^i_{n_im_i} & y^i_{n_i}\end{array} \right]$$
and we denote the entries of $M$ as
$$M=\left[ \begin{array}{ccc|c}
x_{11} & \cdots & x_{1m} & y_{1}\\
\vdots &  & \vdots & \vdots\\
x_{n1} & \cdots & x_{nm} & y_{n}\end{array} \right].$$
We thus assume that $M\Rightarrow_{\lex}\Cub_{n'}$ and we shall prove that either $M_1\Rightarrow_{\lex}\Cub_{n'}$ or $M_2\Rightarrow_{\lex}\Cub_{n'}$. By Theorem~\ref{theorem simple matrix}, we know that there exist $i_1,\dots,i_{n'}\in\{1,\dots,n_1+n_2\}$ such that there does not exist $j\in\{1,\dots,m_1\times m_2\}$ for which $x_{i_aj}=y_{i_a}$ for each $a\in\{1,\dots,n'\}$. Let us denote by $S_1$ the set $S_1=\{a\in\{1,\dots,n'\}\,|\,i_a\in\{1,\dots,n_1\}\}$ and by $S_2$ the set $S_2=\{1,\dots,n'\}\setminus S_1$. If there exist $j_1\in\{1,\dots,m_1\}$ and $j_2\in\{1,\dots,m_2\}$ such that $x^1_{i_aj_1}=y^1_{i_a}$ for each $a\in S_1$ and $x^2_{(i_a-n_1)j_2}=y^2_{i_a-n_1}$ for each $a\in S_2$, by choosing $j\in\{1,\dots,m_1\times m_2\}$ as the index of the left column of $M$ obtained by superposing the $j_1$-th left column of $M_1$ over the $j_2$-th left column of~$M_2$, one obtains that $x_{i_aj}=y_{i_a}$ for each $a\in\{1,\dots,n'\}$, contradicting our hypothesis. By symmetry, we can therefore assume without loss of generality that there does not exist $j_1\in\{1,\dots,m_1\}$ for which $x^1_{i_aj_1}=y^1_{i_a}$ for each $a\in S_1$. Let us define $i'_a\in\{1,\dots,n_1\}$ for each $a\in\{1,\dots,n'\}$ by
$$i'_a=\begin{cases}i_a & \quad\text{if }a\in S_1\\ 1 & \quad\text{if }a\in S_2.\end{cases}$$
The indices $i'_1,\dots,i'_{n'}\in\{1,\dots,n_1\}$ satisfy condition~\ref{theorem simple matrix existence of rows} of Theorem~\ref{theorem simple matrix} and therefore one has $M_1\Rightarrow_{\lex}\Cub_{n'}$.
\end{proof}

\subsection*{Results for general matrices}

We now tackle the question to describe when $M\Rightarrow_{\reg}\Cub_n$ for a (not necessarily simple) extended matrix~$M$. In general, we do not yet know an algorithm to decide whether a general matrix condition implies another one. In the following theorem, we thus have to use another technique than the one used to prove Theorem~\ref{theorem simple matrix}. We will use here, in addition to the results of~\cite{Oprsal}, the embedding theorems of~\cite{Jacqmin3}. In order to do so, we will need the axiom of universes~\cite{AGV}, which will only be used to prove the implication \ref{theorem general matrix reg ess alg}$\Rightarrow$\ref{theorem general matrix reg} (the equivalences \ref{theorem general matrix reg ess alg}$\Leftrightarrow$\ref{theorem general matrix alg}$\Leftrightarrow$\ref{theorem general matrix non existence of algebra} do not rely on the axiom of universes).

\begin{theorem}\label{theorem general matrix}
Let $M$ be an extended matrix (with parameters $n\geqslant 1$, $m\geqslant 0$, $m'\geqslant 0$ and $k\geqslant l \geqslant 0$)
$$M=\left[ \begin{array}{ccc|ccc}
x_{11} & \cdots & x_{1m} & y_{11} & \cdots & y_{1m'}\\
\vdots &  & \vdots & \vdots && \vdots\\
x_{n1} & \cdots & x_{nm} & y_{n1} & \cdots & y_{nm'} \end{array} \right]$$
and let $n'\geqslant 2$ be an integer. The following statements are equivalent:
\begin{enumerate}[label=(\alph*), ref=(\alph*)]
\item\label{theorem general matrix reg} $M\Rightarrow_{\reg}\Cub_{n'}$
\item\label{theorem general matrix reg ess alg} $M\Rightarrow_{\regessalg}\Cub_{n'}$
\item\label{theorem general matrix alg} $M\Rightarrow_{\alg}\Cub_{n'}$
\item\label{theorem general matrix non existence of algebra} There do \emph{not} exist functions $$p_1,\dots,p_{m'}\colon \{0,1\}^m\to\{0,1\}$$ and $$q_1,\dots,q_{k-l}\colon\{0,1\}^l\to\{0,1\}$$ making $A=(\{0,1\},p_1,\dots,p_{m'},q_1,\dots,q_{k-l})$ an algebra of $\VV_M$ such that the $n'$-ary relation $R_{n'}=\{0,1\}^{n'}\setminus\{(0,\dots,0)\}$ is a homomorphic relation on~$A$.
\end{enumerate}
\end{theorem}

\begin{proof}
The equivalence \ref{theorem general matrix alg}$\Leftrightarrow$\ref{theorem general matrix non existence of algebra} is an immediate application of Lemma~3.5 and Proposition~7.7 of~\cite{Oprsal} applied to the variety~$\VV_M$. Under the axiom of universes, the implication \ref{theorem general matrix reg ess alg}$\Rightarrow$\ref{theorem general matrix reg} follows immediately from Theorem~4.6 of~\cite{Jacqmin3}. The implications \ref{theorem general matrix reg}$\Rightarrow$\ref{theorem general matrix reg ess alg}$\Rightarrow$\ref{theorem general matrix alg} being trivial, it remains to prove \ref{theorem general matrix non existence of algebra}$\Rightarrow$\ref{theorem general matrix reg ess alg}. Let us thus assume that \ref{theorem general matrix non existence of algebra} holds and let us consider an essentially algebraic theory $\Gamma=(S,\Sigma,E,\Sigma_t,\Def)$ such that $\Mod(\Gamma)$ is a regular category with $M$-closed relations. We shall prove that it has $\Cub_{n'}$-closed relations. By Theorem~3.3 in~\cite{Jacqmin3}, we know that, for each sort $s\in S$, there exist in~$\Gamma$:
\begin{itemize}
\item for each $j\in\{1,\dots,m'\}$, a term $p^s_j\colon s^m\to s$,
\item for each $a\in\{1,\dots,k-l\}$, a term $q^s_a\colon s^l\to s$
\end{itemize}
such that, for each $i\in\{1,\dots,n\}$ and each $j\in\{1,\dots,m'\}$,
\begin{itemize}
\item if $y_{ij}=x_a\in\{x_1,\dots,x_l\}$, the term $p^s_j(x_{i1},\dots,x_{im})\colon s^l\to s$ (in the variables $x_1,\dots,x_l$) is everywhere-defined and equal to~$x_a$,
\item if $y_{ij}=x_a\in\{x_{l+1},\dots,x_k\}$, the term $p^s_j(x_{i1},\dots,x_{im})\colon s^l\to s$ (in the variables $x_1,\dots,x_l$) is defined for an $l$-tuple $(b_1,\dots,b_l)\in B_s^l$ in a $\Gamma$-model~$B$ if and only if $q^s_{a-l}(b_1,\dots,b_l)$ is defined, and in that case they are equal.
\end{itemize}
Let now $T$ be a homomorphic $n'$-ary relation on the $\Gamma$-model~$B$. Let $s\in S$ and let $b_0,b_1\in B_s$ be such that any $n'$-tuple of elements of $\{b_0,b_1\}$, except maybe $(b_0,\dots,b_0)$, is in~$T_s$. We shall prove that $(b_0,\dots,b_0)\in T_s$. For each $j\in\{1,\dots,m\}$, let us define the function $p_j\colon \{0,1\}^m\to \{0,1\}$ on $(c_1,\dots,c_m)\in\{0,1\}^m$ by $p_j(c_1,\dots,c_m)=0$ if $p^s_j(b_{c_1},\dots,b_{c_m})$ is defined and equal to~$b_0$, otherwise we set $p_j(c_1,\dots,c_m)=1$. Similarly, for each $a\in\{1,\dots,k-l\}$, let us define the function $q_a\colon \{0,1\}^l\to \{0,1\}$ on $(c_1,\dots,c_l)\in\{0,1\}^l$ by $q_a(c_1,\dots,c_l)=0$ if $q^s_a(b_{c_1},\dots,b_{c_m})$ is defined and equal to~$b_0$, otherwise we set $q_a(c_1,\dots,c_m)=1$. It is easy to see that $A=(\{0,1\},p_1,\dots,p_{m'},q_1,\dots,q_{k-l})$ forms an algebra of~$\VV_M$. Let us consider the bijection $f\colon\{0,1\}\to\{b_0,b_1\}$ defined by $f(0)=b_0$ and $f(1)=b_1$ and the induced bijection $f^{n'}\colon\{0,1\}^{n'}\to \{b_0,b_1\}^{n'}$. Since \ref{theorem general matrix non existence of algebra} holds, either there exist $j\in\{1,\dots,m'\}$ and elements $r_1,\dots,r_m\in R_{n'}$ such that $p^{n'}_j(r_1,\dots,r_m)=(0,\dots,0)$ or there exist $a\in\{1,\dots,k-l\}$ and elements $r_1,\dots,r_l\in R_{n'}$ such that $q^{n'}_a(r_1,\dots,r_l)=(0,\dots,0)$ (where $p^{n'}_j$ and $q^{n'}_a$ are the operations induced on $\{0,1\}^{n'}$ by $p_j$ and $q_a$ respectively). We thus have either that $(p^s_j)^{n'}(f(r_1),\dots,f(r_m))$ is defined and equal to $(b_0,\dots,b_0)$, or that $(q^s_a)^{n'}(f(r_1),\dots,f(r_l))$ is defined and equal to $(b_0,\dots,b_0)$. Since $T$ is a homomorphic relation on $B$ and either $f(r_1),\dots,f(r_m)\in T_s$ or $f(r_1),\dots,f(r_l)\in T_s$, we can conclude in both cases that $(b_0,\dots,b_0)\in T_s$.
\end{proof}

Let us notice that condition~\ref{theorem general matrix non existence of algebra} provides a (finite time) algorithm to decide whether $M\Rightarrow_{\reg}\Cub_{n'}$ (or equivalently $M\Rightarrow_{\alg}\Cub_{n'}$), but it seems this is not a polynomial-time algorithm.

Again, since this is our most interesting case, let us emphasize the case $n'=2$ of Theorem~\ref{theorem general matrix}.

\begin{corollary}
For an extended matrix~$M$, the following statements are equivalent:
\begin{enumerate}[label=(\alph*), ref=(\alph*)]
\item Every regular category with $M$-closed relations is a Mal'tsev category.
\item Every variety with $M$-closed relations is a Mal'tsev variety.
\end{enumerate}
\end{corollary}



\vspace{5mm}
\noindent
Michael Hoefnagel\\
Mathematics Division\\
Department of Mathematical Sciences\\
Stellenbosch University\\
Private Bag X1 Matieland 7602\\
South Africa
\vspace{3mm}

\noindent
National Institute for Theoretical and Computational Sciences (NITheCS)\\
South Africa
\vspace{3mm}

\noindent
mhoefnagel@sun.ac.za\\
\vspace{10mm}

\noindent
Pierre-Alain Jacqmin\\
Institut de Recherche en Math\'ematique et Physique\\
Universit\'e catholique de Louvain\\
Chemin du Cyclotron~2\\
B 1348 Louvain-la-Neuve\\
Belgium
\vspace{3mm}

\noindent
Centre for Experimental Mathematics\\
Department of Mathematical Sciences\\
Stellenbosch University\\
Private Bag X1 Matieland 7602\\
South Africa
\vspace{3mm}

\noindent
pierre-alain.jacqmin@uclouvain.be

\end{document}